\documentclass[final,leqno,onefignum,onetabnum]{siamltex}

\usepackage{subfigure,upgreek,dsfont,braket,amssymb,amsmath}
\usepackage[small,bf]{caption}
\usepackage{booktabs,dsfont}
\usepackage{algorithm}
\usepackage{float}
\usepackage{graphicx}
\usepackage{algorithmicx}
\usepackage{color,float}
\usepackage{url}
\usepackage{enumerate}

\usepackage{amssymb, amsmath}

\newcommand{\al}{\alpha}
\newcommand{\be}{\beta}
\newcommand{\ga}{\gamma}
\newcommand{\de}{\delta}
\newcommand{\la}{\lambda}

\newcommand{\auskommentieren}[1]{}

\newcommand{\beq}{\begin{equation}}
\newcommand{\eeq}{\end{equation}}

\newtheorem{remark}[theorem]{Remark}

\DeclareMathOperator{\sgn}{sgn}
\DeclareMathOperator{\dist}{dist}
\DeclareMathOperator{\dive}{div}
\title{Numerical approximation of level set power mean curvature flow}

\author{Axel Kr\"oner\thanks{INRIA Saclay and CMAP, \'Ecole Polytechnique, Route de Saclay, 91128~Palaiseau cedex, France, Tel. +33 (0) 1 69 33 4624, {\tt axel.kroener@inria.fr}}
\and  Eva Kr\"oner\thanks{Department of Crop Science, Georg-August-Universit\"at G\"ottingen,
B\"usgenweg 2, 37077~G\"ottingen, Germany, Tel. +49 (0) 551 3912294, {\tt ekroene@gwdg.de}}
\and Heiko Kr\"oner\thanks{Fachbereich Mathematik, Universit\"at Hamburg, Bundesstra\ss e 55, 
20146 Hamburg, Germany, Tel. +49 (0)40 42838 4942,
{\tt  heiko.kroener@uni-hamburg.de}}}

\begin{document}
\maketitle
\slugger{mms}{xxxx}{xx}{x}{x--x}

\begin{abstract}
In this paper we investigate the numerical approximation of a variant of the mean curvature flow. 
We consider the evolution of hypersurfaces with normal speed given by $H^k$, $k \ge 1$, where $H$
denotes the mean curvature.
We use a level set formulation of this flow and discretize the regularized level set equation with finite elements. 
In a previous paper we proved an a priori estimate for the approximation error between the finite element solution and the solution of the original level set equation. We obtained an upper bound for this error which is polynomial in the discretization parameter and the reciprocal regularization parameter.
The aim of the present paper is the numerical study of the behavior of the evolution and the numerical verification of certain convergence rates. We restrict the consideration to the case that the level set function depends on two variables, 
i.e. the moving hypersurfaces are curves. Furthermore, we confirm for specific initial curves and 
different values of $k$ that the flow improves the 
isoperimetrical deficit. 
\end{abstract}

\begin{keywords}
geometric evolution equations, level set formulation, viscosity solution, finite elements
\end{keywords}

\begin{AMS}53A10, 65L60, 35D40\end{AMS}

\pagestyle{myheadings}
\thispagestyle{plain}
\markboth{Numerical approximation of level set power mean curvature flow}{Axel Kr\"oner,
Eva Kr\"oner, Heiko
Kr\"oner}

\section{Introduction}

Geometric evolution equations, especially cur\-va\-ture-dependent interface motion has been studied for many 
years in both pure and applied mathematics. One application example is the evolution of soap films
and the behavior of the boundaries of oil drops on a surface of water which evolve into disks.
In material science, for example, the evolving surfaces 
might be grain boundaries in alloys which separate differing orientations of the same 
crystalline phase. In image processing, for example, one wants to identify a dark shape in a light background 
in a two-dimensional image. Therefore
a so-called snake contour is evolved so that it wraps around the shape. We refer to \cite{DDE} for a more detailed 
exposition of these applications and, e.g.,
\cite{DDE, CCCD, CV, MS} and references therein for further applications. 

The most famous case for such an interface motion is the mean curvature flow of closed $n$-dimensional hypersurfaces in Euclidean space $\mathbb{R}^{n+1}$ or as special case the curve shortening flow of closed curves in the plane. 
Under this flow the hypersurface moves in normal direction so that the normal speed equals the
(mean) curvature and a convex initial hypersurface shrinks to a round point in finite time, where in case of closed, embedded plane curves this even holds if the 
convexity assumption is left out, cf. \cite{HI2} and \cite{Gage, Gage_Hamilton, Grayson}.

Mean curvature flow can be formulated in parametric form, where the moving hypersurface is given by a parametrization over a fixed hypersurface which depends on the evolution time as variable, cf. \cite{HI2}; a special case is graphical mean curvature flow, where the hypersurface is given as the graph of a 
height function, cf. \cite{EH}. A third possibility is a phase field approach to mean curvature flow, cf. \cite{NV} and references therein,  and the fourth which will be considered in the following in more detail is to consider the PDE resulting from the level set formulation. The level set formulation is powerful because it can handle topological changes of the moving hypersurface. Level set methods were introduced by Sethian and Osher, see \cite{OS, Sethian, Giga}, and have been applied to a wide range of problems.

In the present paper we are concerned with the level set formulation of a modified version of the mean curvature flow. Instead of the mean curvature we prescribe the normal speed of the evolution to be a power $k\ge 1$ of the mean curvature and assume that the initial hypersurface has positive mean curvature. In our previous paper \cite{K} we proved a priori error estimates for a finite element approximation of this flow. The aim of the present paper is the numerical study of the behavior of the evolution and the numerical verification of certain convergence rates. We restrict the consideration to the case that the level set function depends on two variables, i.e. the moving hypersurfaces are curves. Furthermore, we confirm for specific initial curves 
and different values of $k$ that the flow improves the 'isoperimetrical deficit'. 

To specify how our flow looks like we give the following parametric formulation of this flow.
Let  $M$ be a smooth $n$-dimensional compact manifold without boundary (at the moment it is sufficient to assume only $k>0$)  and $x_0:M \rightarrow \mathbb{R}^{n+1}$ a smooth embedding such that $M_0=x_0(M)$ has positive mean curvature, then we consider a solution of the following fully nonlinear parabolic initial value problem. 
Find $T>0$ and a smooth mapping
\beq
x: [0, T) \times M \rightarrow \mathbb{R}^n
\eeq
with
\begin{equation}
\begin{aligned} 
\label{classical_pmcf}
x(0, \cdot) &= x_0, \\
\frac{d}{dt}x(t,\xi) &= -H^k\nu.
\end{aligned}
\end{equation}
Here, $H$ and $\nu$ denote the mean curvature and the outer normal of $x(t, M)$ at $x(t, \xi)$, respectively.
We call this a power mean curvature flow (PMCF).

Why is this flow interesting and how does this flow behave? This flow has been considered in a series of papers under different aspects. In \cite{S} it is shown that the flow (\ref{classical_pmcf}) exists on a maximal, finite time interval and that, approaching the final time, the surfaces contract to a point. In \cite{SS} the flow is considered in the case $k \ge 1$. It is shown that if initially the ratio of the biggest and smallest principal curvature at every point is close enough to 1, depending only on $k$ and the dimension $n$ of the hypersurfaces, then this is maintained under the flow. As a consequence the authors of \cite{SS} obtain that, when rescaling appropriately as the flow contracts to a point, the evolving surfaces converge to the unit sphere. The paper \cite{S2}  shows that for $k \ge n-1$ the flow improves a certain 'isoperimetrical difference'. As singularities may develop before the volume goes to zero, a weak level-set formulation for such flows is developed and it is shown that the monotonicity of
the isoperimetrical difference is still valid. This proves the isoperimetrical inequality for $n \le 7$. A further reason which makes this flow interesting is that mean curvature flow is used in denoising images, cf. \cite{AGLM} for such applications, and we 
expect that PMCF with the possibility to choose different values for $k \ge 1$ is an interesting alternative for this purpose. We mention the remarks in \cite{M} which state that in the case $0<k\le 1$ equation (\ref{classical_pmcf}) plays a key role in the context of image processing. 

As announced above we will perform our calculation for curves, so it is of interest to give references for this case, both of theoretical and practical nature.
In \cite{BGN, BGN2} a finite element approximation of the parametric formulation  of the flow (\ref{classical_pmcf}) in the case of curves is formulated and 
stability bounds are derived, see also \cite{BGN3}. In \cite{MSe} the evolution of plane curves driven by a nonlinear function of curvature and anisotropy is considered with a focus on the analysis of the parametric formulation of such a flow, see also \cite{Andrews, A03}. The analysis of boundaries of shapes in the context of morphological and shape image processing leads to an equation of the form (\ref{classical_pmcf}) in the case of curves. This has been introduced in \cite{AGLM,AM,ST} and we mention especially the case $k=\frac{1}{3}$ which is the so-called affine curvature equation, cf. \cite{AST} and see also (\ref{pmcf_time_dependent}) and the following text lines.
In summary one can say that our flow (\ref{classical_pmcf}) (in the case $k \ge 1$) plays
an important role for applications and has received a lot attention so far but the numerical approximation of 
its level set formulation has apart from our own previous work \cite{K}  not been analyzed yet.

We introduce our notation. The Euclidean norm of $\mathbb{R}^n$ is denoted by $|\cdot|$. For an open
subset $\Omega$ of $\mathbb{R}^n$  and $m\in \mathbb{N}^{*}$, $p\ge 1$ we denote the 
corresponding Sobolev spaces by $W^{m,p}(\Omega)$, $W^{m,p}_0(\Omega)$,
$H^m(\Omega)=W^{m,2}(\Omega)$ and $H^m_0(\Omega)=W^{m,2}_0(\Omega)$. The dual spaces are denoted by
$W^{-m,p}(\Omega)=W^{m,p}_0(\Omega)^{*}$ and the dual pairing 
by 
\beq
W^{-m,p}(\Omega) \times W^{m,p}_0(\Omega)\ni (F, \varphi) \mapsto \left<F, \varphi\right>=F\varphi
\in \mathbb{R}.
\eeq

In the following we will introduce a (stationary) level set formulation for (\ref{classical_pmcf}) and start for this purpose by recalling the (time-dependent) level set formulation for the mean curvature flow, i.e. the case $k=1$ in (\ref{classical_pmcf}). Let $M_0\subset \mathbb{R}^{n+1}$ be a given initial hypersurface and choose a continuous function $u_0: \mathbb{R}^{n+1}\rightarrow \mathbb{R}$ such that 
\beq
M_0 = \{x \in \mathbb{R}^{n+1}: u_0(x)=0\}.
\eeq
If $u : [0, \infty)\times \mathbb{R}^{n+1}\rightarrow \mathbb{R}$ is the unique viscosity solution of
\begin{equation} \label{mcf_time_dependent}
\begin{aligned} 
\frac{d}{dt}u =& |Du|  \dive  \left(\frac{Du}{|Du|}\right) 
= \mathcal{H}(Du, D^2u)
\end{aligned}
\end{equation}
in $\mathbb{R}^{n+1}\times (0, \infty)$ with $u(0, \cdot) = u_0$ in $\mathbb{R}^{n+1}$, where
\beq
\mathcal{H}(p, Y) = \sum_{i,j}\left(\de_{ij}-\frac{p_ip_j}{|p|^2}\right)Y_{ ij} 
\eeq 
for $p=(p_i)\in \mathbb{R}^n\setminus\{0\}$ and $Y=(Y_{ij})\in \mathbb{R}^{n\times n}$.
We call the family of the
\beq
M(t) = \{x \in \mathbb{R}^{n+1}: u(t,x)=0\}, \quad t >0,
\eeq
a (time dependent) level set mean curvature flow. Equation (\ref{mcf_time_dependent}) is a quasilinear, degenerate and possibly singular (if $Du=0$) parabolic equation. Existence and uniqueness of a solution for this equation is proved in \cite{CGG, CGG1, ES}.

If we include our nonlinearity $H^k$ of the mean curvature in this formulation we get instead  of (\ref{mcf_time_dependent}) the fully nonlinear, degenerate and possibly singular parabolic equation
\begin{equation} \label{pmcf_time_dependent}
\begin{aligned} 
\frac{d}{dt}u =& |Du| \left(\dive  \left(\frac{Du}{|Du|}\right)\right)^k 
=& \left(\sum_{i,j}\left(\de_{ij}-\frac{D_iuD_ju}{|Du|^2}\right)D_iD_j u\right)^k|Du|^{1-k}
\end{aligned}
\end{equation}
To our knowledge an existence proof for (\ref{pmcf_time_dependent}) is only known for the case $0<k\le 1$, cf. \cite{M}
and the references therein.
But the case under consideration in the present paper is $k>1$.  In case $k>1$ the proof 
presented in \cite{M} does not work any more because the linear growth of the elliptic part of 
the operator needed to apply classical arguments is not available. 

The time dependent formulation \ref{pmcf_time_dependent} in the $k=\frac{1}{3}$ case, i.e. the affine curvature
equation, is used for image processing, cf. \cite{AGLM, Giga}.
In \cite{M}  equation (\ref{pmcf_time_dependent}) in case $0<k\le 1$ is approximated by a family of regularized equations and 
rates of convergence of the corresponding solutions are obtained.

In the case $k=1$, of course, we have mean curvature flow, and the corresponding equation (\ref{mcf_time_dependent}) 
has been studied intensively analytically and numerically, cf., e.g.,
\cite{CarliniFalconeFerretti:2010,CL,DDE,KohnSerfaty:2006}.
We want to point out the paper \cite{D} by Deckelnick, where the solution $u^{\varepsilon}$ of a regularized version of (\ref{mcf_time_dependent}) 
is approximated by a finite difference scheme which was originally proposed by Crandall and Lions \cite{CL}. 
In Deckelnick's paper rates for the convergence of the discrete solution to the solution $u$ of the (not regularized) 
level set equation are proved. 
The total error consists of a regularization error of the form 
\beq \label{1000}
\|u-u^{\varepsilon}\|_{L^{\infty}(\Omega)} \le c_{\al}\varepsilon^{\al}
\eeq
with $\al\in (0, \frac{1}{2})$ arbitrary and $c_{\al}$ a positive constant, see \cite[Theorem 1.2]{D} for details, and a discretization error which
is a {\it polynomial} expression in the numerical parameter and the reciprocal regularization 
parameter. Furthermore, the concrete value for the convergence order of the discretization error
(and hence for the total approximation error) is very low; 
the main point here is that this 
rate is of polynomial order. 

This is not self-evident as can be seen in the paper \cite{Deckelnick_Dziuk_2001}. 
There the viscosity solution $u$ of (\ref{mcf_time_dependent}) is approximated by a solution $u_{\varepsilon}$ of the 
regularized equation and then the regularized equation is approximated by a solution $u_{\varepsilon, h}$ of a semi 
discrete problem. The regularization error is again of the form (\ref{1000}) but the 
error $u_{\varepsilon}-u_{\varepsilon, h}$ measured in a certain energy norm, cf. \cite[Theorem 6.4]{DDE}, is 
only of order $c_{\varepsilon}h$, where, and this is the important point, the constant $c_{\varepsilon}$ depends 
{\it exponentially} on $\frac{1}{\varepsilon}$. Numerical tests as written there, however, suggest that the resulting bound 
overestimates the error. In the special case of two dimensions, i.e. the moving hypersurfaces are curves,
Deckelnick and Dziuk \cite{Deckelnick_Dziuk_2001} prove $L^{\infty}$-convergence (without rates) of the 
discrete solution provided $h=h(\varepsilon)$ sufficiently small, where 'sufficiently small' is not given by an 
explicit formula or polynomial dependence.

Let us now consider the case $k \ge 1$ which is the relevant one in the present paper. To circumvent the above mentioned problem with the growth in the elliptic part of the operator if $k>1$ Schulze \cite{S} uses a stationary level set formulation at which the nonlinearity due to the exponent $k$ affects only lower order terms.
We present Schulze's stationary level set formulation of the PMCF. 

Let $\Omega \subset \mathbb{R}^{n+1}$ be open, connected and bounded having smooth boundary $\partial \Omega$ with positive mean curvature.
 Here, $\partial \Omega$ plays the role of the initial hypersurface. 
 We call the level sets $\Gamma_t =\partial\{x\in \Omega: u(x)>t\}$, $t\ge 0$, of the continuous function $0 \le u \in C^0(\bar \Omega)$ a (stationary) level set PMCF, if $u$ is a viscosity solution of 
\begin{equation}
\begin{aligned} 
 \label{levelset_pmcf}
\dive \left( \frac{D u}{|D u|} \right) =& -\frac{1}{|D u|^{\frac{1}{k}}} && \quad \text{in } \Omega \\
u=&0 && \quad \text{on } \partial \Omega.
\end{aligned} 
\end{equation}
For a definition of a viscosity solution for this equation we refer to \cite[Section 2]{K}.
If $u$ is smooth in a neighborhood of $x\in \Omega$ with non vanishing gradient and satisfies in this neighborhood 
(\ref{levelset_pmcf}), then the level set $\{u=u(x)|x\in \Omega\}$ moves locally at $x$  according to (\ref{classical_pmcf}).
Using elliptic regularization of level set PMCF we obtain the equation
 \begin{equation}
\begin{aligned} \label{regularized_levelset_pmcf}
\dive \left(\frac{D u^{\varepsilon}}{\sqrt{\varepsilon^2+|D u^{\varepsilon}|^2}}\right) &= -(\varepsilon^2
+|D u^{\varepsilon}|^2)^{-\frac{1}{2k}} \quad &&\text{in } \Omega, \\
u^{\varepsilon}&= 0\quad  &&\text{on } \partial \Omega,
\end{aligned} 
\end{equation}
which has a unique smooth solution $u^{\varepsilon}$ for sufficiently small $\varepsilon >0$, cf. \cite[Section 4]{S}; moreover, there is $c_0>0$ such that
\beq \label{12}
\|u^{\varepsilon}\|_{C^1{(\bar \Omega})} \le c_0
\eeq
uniformly in $\varepsilon$ and (for a subsequence) 
\beq \label{15}
u^{\varepsilon} \rightarrow u \in C^{0,1}(\bar \Omega)
\eeq
in $C^0(\bar \Omega)$. We call $u$ a weak solution of  (\ref{levelset_pmcf}), which is unique for $n\le 6$. 
All the above facts are proved in \cite[Section 4]{S} under the assumption that $k \ge 1$. 
A weak solution of (\ref{levelset_pmcf}) satisfies (\ref{levelset_pmcf})  in the viscosity sense, cf. Section \cite[Section 2]{K}.
Furthermore Schulze's existence result is restricted to the case $k \ge 1$.

What looks as a disadvantage at first glance, namely the fact that our level set function does not depend on the time explicitly and hence no explicit 
Euler method is applicable (as for example in Deckelnick's paper \cite{D}), 
has the advantage that we have a divergence structure for the elliptic part of the operator which would not be the case if we would use the time dependent level set formulation (\ref{pmcf_time_dependent}) (in addition we would lack a proof of existence of a solution).

To our knowledge our previous paper \cite{K} is the only numerical analysis result for Schulze's level set 
formulation so far. 
In \cite{K} we proved an explicit rate for the convergence of the solution $u^{\varepsilon}$ of 
(\ref{regularized_levelset_pmcf}) to the solution $u$ of equation (\ref{levelset_pmcf}) which depends on $k$, 
see Section \ref{50}, where we recall the result. 
Using the divergence structure of the elliptic part of (\ref{regularized_levelset_pmcf}) we proved 
existence of a finite element approximation $u^{\varepsilon}_h$ of $u^{\varepsilon}$ and an approximation 
rate. 
Summarized we have a total approximation error 
\beq \label{100}
u-u^{\varepsilon}_h = (u - u^{\varepsilon} )+ (u^{\varepsilon} -u^{\varepsilon}_h)
\eeq 
which consists of the regularization error (first bracket in equation (\ref{100})) and the discretization error 
(second bracket in equation (\ref{100})).
We obtained a {\it polynomial} rate in $\frac{1}{\varepsilon}$ and $h$ for the total
approximation error provided the discretization parameter is sufficiently small compared with the 
regularization parameter (the coupling between the discretization and the regularization parameter is also of polynomial order). 
The order of convergence is polynomial (in contrast, e.g., to \cite[Theorem 6.4]{DDE}, where the authors 
obtain exponential order of convergence) and comparable to  the one obtained in \cite{D}. 
In both cases despite from being of polynomial order the precise order is rather of theoretical value. 
As in the remarks following \cite[Theorem 6.4]{DDE} stating that experiments indicate that the proven approximation 
error overestimates the real error, we have a similar behavior for our situation, 
cf. Sections \ref{51}, 
\ref{50}, and \ref{52}.

Furthermore, we validate for some examples that the flow improves (i.e. decreases) the isoperimetrical deficit 
\beq \label{isoperimetric_deficit}
A(t)^{\frac{n+1}{n}}-c_{n+1} V(t),
\eeq
where $A(t)$ denotes the $n$-dimensional volume of the evolving hypersurface, $V(t)$ the $(n+1)$-dimensional 
volume of the enclosed subset of $\mathbb{R}^{n+1}$ and $c_{n+1}$ the Euclidean isoperimetrical constant,
cf. \cite{S} for a proof of this property. The isoperimetrical deficit  (\ref{isoperimetric_deficit}) is nonnegative 
and zero if and only if the hypersurface is a sphere.
Since we restrict ourselves to the case $n=1$ 
the evolving hypersurfaces are curves and the variables in the isoperimetrical deficit 
become arc length and enclosed area. And there holds $c_2=4 \pi$. 
Geometrically more interesting is the case $n \ge 2$ because then one can have non convex initial 
hypersurfaces with positive mean curvature which develop topological changes under the flow. 
To validate our level set ansatz with respect to its convergence behavior and numerical properties
restricting ourselves to the curve case seems to be a reasonable and legitimate technical simplification.

The paper is organized as follows. In Section \ref{51} we prove an error estimate for the discretization error
and validate it numerically. In Section \ref{50} we recall and calculate in detail an error estimate
for the regularization error from our previous paper \cite{K} and provide numerical examples. In Section \ref{52} we provide numerical examples 
for the total approximation error. In Section \ref{70} we illustrate with an example the influence of the 
exponent $k$ on the behavior of the flow.

\section{Discretization error} \label{51}
In this section we assume that the space dimension $n+1$ is 2 or 3 and that $\Omega$ is convex.
The latter is only a restriction if $n+1=3$ since $\partial \Omega$ has positive mean curvature 
by our assumptions in the introduction. We fix $\varepsilon>0$ at a small value and approximate the solution $u^{\varepsilon}$ 
of equation 
(\ref{regularized_levelset_pmcf}) by a finite element solution $u^{\varepsilon}_h$ 
which seems to be an appropriate method in view of 
the divergence structure of the operator. The goal is to analyze the discretization error 
$u^{\varepsilon}-u^{\varepsilon}_h$. 

Let $(T_h, \Omega_h)$ be a quasi-uniform triangulation of $\Omega$ with mesh size 
$0<h<h_0$, $h_0$ sufficiently small, and $V_h \subset H^1(\Omega_h)$ 
the finite element space given by
\beq 
V_h = \left\{v\in C^0(\bar \Omega_h): v_{|\partial \Omega_h}=0, \ v_{|T} \text{ linear }
\forall T\in T_h\right\}.
\eeq
In view of the convexity of $\Omega$ there
holds $\Omega_h \subset \Omega$. A function $u_h \in V_h$ will be also considered as a function on $\Omega$ by
extending it by zero in $\Omega\setminus \Omega_h$. Then $v_h\in H^1(\Omega)$. Our variational formulation is given as in our previous paper \cite{K} by
 \begin{equation}
\begin{aligned} \label{FE_levelset_pmcf}
\int_{\Omega_h}\frac{\left<D u^{\varepsilon}_h,D v_h\right>}{\sqrt{\varepsilon^2+|D u^{\varepsilon}_h|^2}} dx=& 
\int_{\Omega_h}(\varepsilon^2+|D u^{\varepsilon}_h|^2)^{-\frac{1}{2k}} v_hdx\quad \forall \ v_h \in V_h. 
\end{aligned} 
\end{equation}
For formal reason we might consider  boundary tetrahedrons (boundary triangles in case $d=2$) 
to be extended to a boundary tetrahedron with one 'curved face'. 
Therefore we will replace a boundary element $T \in T_h$ (i.e. $n+1$ vertexes of $T$ lie on 
$\partial \Omega$) by $\tilde T=T\cup B$ with
\beq
B = \{tp+(1-t)Pp\ | \ 0 \le t \le 1,  p \in F\},
\eeq
where $F$ is the boundary face of $T$, i.e. $n+1$ vertexes of $F$ lie on $\partial \Omega$, 
and $Pp$ is the unique minimizer of $\dist(p, \cdot)_{|\partial \Omega}$. 
We denote the resulting triangulation by $\tilde T_h$. 
This leaves the space of finite element functions we use (namely $V_h$) unchanged.
Note, that the boundary strip $\Omega \setminus \Omega_h$ has measure $O(h^2)$.

A similar equation as (\ref{FE_levelset_pmcf}) is the 
stationary level set formulation for the inverse mean curvature flow which is used in \cite{FNP}.
There also a total approximation error, a discretization error 
and a regularization error  appears. Furthermore, a rate for the discretization error ($O(h)$ for the $H^1$-error and $O(h^2)$ 
for the $L^2$-error) is proved. But the dependence of the constants on the regularization parameter 
which appear in these error estimates is not analyzed theoretically. 
In contrast to \cite{K} in \cite{FNP} no theoretical estimate for the regularization error 
(and hence for the total approximation error) is given, and such a rate seems to be an open 
problem so far, cf. \cite[Remark~4]{FNP}. But this issue is addressed numerically in \cite{FNP} and 
calculations suggest that the regularization error is $O(\varepsilon)$, where $\varepsilon$ here also denotes 
the regularization parameter.

We remark that when we considered the discretization in our previous paper, see \cite[Section 6]{K},
the space $V_h$ consisted of continuous functions on $\Omega_h$
 which are piecewise polynomials of degree $\le 2$ (and not linear as here) and we assumed 
 that $\Omega \subset \mathbb{R}^2$
 and existence of
 a solution of (\ref{FE_levelset_pmcf}) in this case was shown. 
 The reason for this is that in  \cite[Section 6]{K} 
 error bounds for the discretization error are proved which contain the dependence of $\varepsilon$ explicitly. Therefore 
 an estimate for the norm of the inverse of 
 $L_{\varepsilon}$ and its dual $L_{\varepsilon}^{*}$, where $L_{\varepsilon}$
 is the derivative of the regularized differential operator, see (\ref{2000}) for a definition,  
 is calculated via the intermediate step of some rather technical sup-norm estimates and inverse inequalities 
 which make it necessary to consider higher order elements. 
 Since these $\varepsilon$-dependencies do not play a role for the discretization error under consideration 
 in this section we present the proof for the present case in easier form and also for
  $W^{1,p}$-norms in Theorem \ref{60} with general $p\ge n+1$ (contrary to \cite[Section 6]{K}, where $p<4$
  is assumed) which is necessary 
  to prove an optimal $L^2$-error
  estimate, cf.~Theorem \ref{87}.
  
  We start with a definition and properties of the linear operator $L_{\epsilon}$ and its dual.
  
  Let $p>1$. We define for $\varepsilon>0$ and $z \in \mathbb{R}^n$
\beq
|z|_{\varepsilon}:=f_{\varepsilon}(z):=\sqrt{|z|^2+\varepsilon^2}
\eeq
and denote derivatives of $f_{\varepsilon}$ with respect to $z^i$ by $D_{z^i}f_{\varepsilon}$.
There holds
\beq
D_{z^i}f_{\varepsilon}(z)= \frac{z_i}{|z|_{\varepsilon}}, \quad 
D_{z^i}D_{z^j}f_{\varepsilon}(z)=\frac{\de_{ij}}{|z|_{\varepsilon}} - \frac{z_iz_j}{|z|^3_{\varepsilon}}.
\eeq
We define the operator $\Phi_{\varepsilon}$ by
\beq
 \Phi_{\varepsilon}: W^{1,p}_0(\Omega)\rightarrow W^{-1,p^{*}}(\Omega), \quad \Phi_{\varepsilon}(v)=- D_i \left(\frac{D_i v}{|D v|_{\varepsilon}}\right) 
 - \frac{1}{|D v|_{\varepsilon}^{\frac{1}{k}}}, 
\eeq
where $\frac{1}{p}+\frac{1}{p^{*}}=1$, so that (\ref{regularized_levelset_pmcf}) can be written as
\beq
\Phi_{\varepsilon}(u^{\varepsilon})=0.
\eeq
We denote the derivative of $\Phi_{\varepsilon}$ in $u^{\varepsilon}$ by
\beq \label{2000}
L_{\varepsilon}:=D\Phi_{\varepsilon}(u^{\varepsilon})
\eeq
and have for all $\varphi \in W^{1,p}_0(\Omega)$ that
\begin{equation}
\begin{aligned}
L_{\varepsilon}\varphi &= -D_i\left(D_{z^i}D_{z^j}f_{\varepsilon}(D u^{\varepsilon})D_j \varphi\right)
+\frac{1}{k} f_{\varepsilon}(D u^{\varepsilon})^{-1-\frac{1}{k}}D_{z^j}f_{\varepsilon}(D u^{\varepsilon})D_j \varphi \\
& =: -D_i(a^{ij}D_j\varphi) + b^iD_i \varphi,
\end{aligned}
\end{equation}
where we use the convention to sum over repeated indices.
The coefficients $a^{ij}$ and $b^i$ are in $C^{\infty}(\bar \Omega)$. Note, that the estimate 
(\ref{12}) is not available for higher order derivatives of $u^{\varepsilon}$
but since we fix $\varepsilon$ in the present section, this does not have an effect
on the following considerations.

The linear operator
\beq
L_{\epsilon}: W^{1,p}_0(\Omega)\rightarrow W^{-1,p^{*}}(\Omega)
\eeq
and its adjoint operator $L_{\epsilon}^{*}$ are topological isomorphism, cf. Corollary \ref{5004}
in the Appendix.

We get from  \cite[Theorem 8.5.3]{BS} for $L=L_{\epsilon}$ or $L=L_{\epsilon}^{*}$ and $F\in W^{-1,p^{*}}(\Omega)$ that there is a unique solution $u_h\in V_h$ of
\beq
\left<Lu_h, \varphi_h\right> = F\varphi_h \quad \forall \varphi_h \in V_h,
\eeq
where $u\in H^1(\Omega)$ is the unique solution of $Lu=F$
and there holds the estimate
\beq
\|u_h\|_{W^{1,p}(\Omega)}+ \|u-u_h\|_{W^{1,p}(\Omega)} \le c\|u\|_{W^{1,p}(\Omega)}. 
\eeq
Furthermore, if $F\in L^p(\Omega)$ we have
\beq
\|u-u_h\|_{W^{1,p}(\Omega)} +h\|u-u_h\|_{L^p(\Omega)} \le ch^2 \|F\|_{L^p(\Omega)}.
\eeq

\begin{remark} \rm
Note, that we actually used the assertion of \cite[Theorem 8.5.3]{BS} with slightly different assumptions. The difference from the assumptions we need to the one assumed in \cite[Theorem 8.5.3]{BS} are as follows.
 \begin{enumerate}[(i)]
 \item We assume a right-hand side $F\in W^{-1,p^{*}}(\Omega)$ (instead $F\in L^p(\Omega)$).
 \item We consider the equation on $\Omega$ (instead of a polygonal domain) and use as discretization the triple $(\tilde T_h, \Omega, V_h)$.
 \end{enumerate}
 \end{remark}

There holds the following theorem.
\begin{theorem} \label{60}
For every $p>n+1$ and small $h>0$ there exists a constant $0<c=c(\|u^{\varepsilon}\|_{W^{2, 2}(\Omega)},p)$ such that (\ref{FE_levelset_pmcf}) has a solution $u^{\varepsilon}_h \in V_h$ satisfying
\beq \label{1002}
\|u^{\varepsilon}-u^{\varepsilon}_h\|_{W^{1,p}(\Omega)} \le c h.
\eeq
This solution is unique in a small $W^{1,p}$-neighborhood of $u^{\varepsilon}$ in $V_h$.
\end{theorem}
\begin{proof}
The proof of this lemma  is adapted from \cite[Section 6]{K}, where the result is proved for $p<4$
and quadratic finite elements. 

We set
\beq
\bar B^h_{\rho} = \{v_h \in V_h: \|u^{\varepsilon}-v_h\|_{W^{1,p}(\Omega)}\le \rho\},
\eeq 
where we choose
\beq \label{71}
\rho = h^{\la} 
\eeq
for an arbitrary and now fixed $\frac{n+1}{p}<\lambda<1$.

We will obtain $u_h$ as the unique fixed point in $ \bar B^h_{\rho}$ of the operator
$T: V_h\rightarrow V_h$ with
\beq \label{204}
L_{\varepsilon} (w_h-Tw_h) = \Phi_{\varepsilon}(w_h), \quad w_h \in V_h.
\eeq
We show that $\bar B^h_{\rho}\neq \emptyset$, that $T$ is a contraction and that $T(\bar B^h_{\rho})\subset \bar B^h_{\rho}$.

(i) Let $I_hu^{\varepsilon}$ be the interpolation of $u^{\varepsilon}$, i.e. the continuous piecewise linear function on $\Omega_h$ which is equal to $u^{\varepsilon}$ at all nodes of $\Omega_h$. We extend 
$I_hu^{\varepsilon}$ by zero to a function on $\Omega$.
In view of
\beq
\|I_hu^{\epsilon} - u^{\epsilon} \|_{W^{1,p}(\Omega)} \le ch
\eeq
we have $I_h u^{\varepsilon} \in \bar B^h_{\rho}$ for small $h$.

(ii) Let $v_h, w_h \in \bar B^{h}_{\rho}$, $\xi_h = v_h-w_h$ then using (\ref{204}) we conclude

\begin{equation} \label{501}
 \begin{aligned}
L_{\varepsilon} (Tv_h-Tw_h)  &= L_{\varepsilon}\xi_h + \Phi_{\varepsilon}(w_h)-\Phi_{\varepsilon}(v_h) \\
&= (L_{\varepsilon}-D\Phi_{\varepsilon}(v_h + \Theta \xi_h))\xi_h \\
&=: F
\end{aligned}
\end{equation}
with a $\Theta \in (0,1)$. In order to estimate $\|F\|_{W^{-1,p^{*}}}(\Omega)$ which leads to
an estimate of $\|T{v_h}-Tw_h\|_{W^{1,p}(\Omega)}$ in view of Corollary \ref{5004}
we choose $\psi \in W^{1,p^{*}}_0(\Omega)$ with $\|\psi\|_{W^{1,p^{*}}(\Omega)}\le 1$ and 
estimate 
$
\left<F, \psi\right>.
$
To do so we use a mean value theorem for which we need the following auxiliary estimate \begin{equation}
 \begin{aligned}
 \|D u^{\varepsilon}-(Dv_h &+ \Theta D\xi_h)\|_{L^{\infty}(\Omega)} \\
 & \le \|Du^{\varepsilon}-DI_hu^{\varepsilon}
 \|_{L^{\infty}(\Omega)} + \|D I_h u^{\varepsilon}-D\tilde v_h\|_{L^{\infty}(\Omega)} \\
 & \le c h + c\rho h^{-\frac{n+1}{p}},
 \end{aligned}
\end{equation}
where $\tilde v_h=v_h + \Theta \xi_h\in \bar B_{\rho}^h$ and where we used an inverse estimate.
The resulting estimate implies
\begin{equation}
\begin{aligned}
\|T{v_h}-Tw_h\|_{W^{1,p}(\Omega)} \le&  c (h + \rho h^{-\frac{n+1}{p}})\|\xi_h\|_{W^{1,p}(\Omega)} \\
\le & \frac{1}{4}\|\xi_h\|_{W^{1,p}(\Omega)} 
\end{aligned}
\end{equation}
for small $h$.

(iii) Let $w_h \in \bar B_{\rho}^h$. There holds
\begin{equation}
\begin{aligned}
 \|Tw_h -&u^{\varepsilon}\|_{W^{1,p}(\Omega)} \\
  \le& \|Tw_h-TI_h u^{\varepsilon}\|_{W^{1,p}(\Omega)}
  + \|TI_h u^{\varepsilon}-I_h u^{\varepsilon}\|_{W^{1,p}(\Omega)} \\
  & + 
  \|I_h u^{\varepsilon}-u^{\varepsilon}\|_{W^{1,p}(\Omega)} \\
  \le & \frac{\rho}{2} + \|TI_h u^{\varepsilon}-I_h u^{\varepsilon}\|_{W^{1,p}(\Omega)} + c h
\end{aligned}
\end{equation}
It remains to estimate the norm on the right-hand side. There holds
\begin{equation}
 \begin{aligned}
 \|TI_h u^{\varepsilon}-I_h u^{\varepsilon}\|_{W^{1,p}(\Omega)} \le& c \|\Phi_{\varepsilon}(I_h u^{\varepsilon})\|
 _{W^{-1,p^{*}}(\Omega)} \\
 = &c\|\Phi_{\varepsilon}(I_h u^{\varepsilon})-\Phi_{\varepsilon}(u^{\varepsilon})\| 
 _{W^{-1,p^{*}}(\Omega)} \\
 \le & c h
 \end{aligned}
\end{equation}
again by a mean value theorem estimate. In view of (\ref{71}) there holds
\beq
T(\bar B_{\rho}^h) \subset \bar B_{\rho}^h.
\eeq
\end{proof}


In the following theorem we improve the $L^p$-error estimate of Theorem \ref{60}, therefore we use a 
duality 
argument as in \cite{FNP}.
\begin{theorem} \label{87}
For $p>n+1$ there holds
\beq \label{88}
\|u^{\varepsilon}-u^{\varepsilon}_h\|_{L^p(\Omega)} \le ch^2
\eeq
with $c=c(\|u^{\varepsilon}\|_{W^{2, 2}(\Omega)},p)>0$.
\end{theorem}
\begin{proof}
 From the definitions of $u^{\varepsilon}$ and $u^{\varepsilon}_h$ we get for all $\varphi_h \in V_h$
 \begin{equation}
\begin{aligned}
 \int_{\Omega}\left(\frac{\nabla u^{\varepsilon}}{|\nabla u^{\varepsilon}|_{\varepsilon}}
 -\frac{\nabla u^{\varepsilon}_h}{|\nabla u^{\varepsilon}_h|_{\varepsilon}}\right) \cdot \nabla
 \varphi_h dx 
 +\int_{\Omega}\left(|\nabla u^{\varepsilon}|^{\frac{1}{k}}_{\varepsilon}
 -|\nabla u^{\varepsilon}_h|^{\frac{1}{k}}_{\varepsilon}\right)
\varphi_h dx =0
 \end{aligned}
 \end{equation}
 This equation can be written as
 \beq \label{72}
 \int_{\Omega}\left(A^{\varepsilon}_h\nabla e^{\varepsilon}_h\right) \cdot \nabla \varphi_h dx + 
 \int_{\Omega}\left(a^{\varepsilon}_h\cdot \nabla e^{\varepsilon}_h\right) \varphi_h dx = 0
 \eeq
 with
 \begin{equation}
  \begin{aligned}
   A^{\varepsilon}_h =& \int_0^1D^2f_{\varepsilon}(\nabla u^{\varepsilon}+t\nabla(u^{\varepsilon}_h-u^{\varepsilon}))dt \\
   a^{\varepsilon}_h=& \frac{1}{k}\int_0^1f_{\varepsilon}(\nabla u^{\varepsilon}+t\nabla(u^{\varepsilon}_h-
   u^{\varepsilon}))^{\frac{1}{k}-1}
   Df_{\varepsilon}(\nabla u^{\varepsilon}+t\nabla(u^{\varepsilon}_h-u^{\varepsilon}))
   dt \\
   e^{\varepsilon}_h =& u^{\varepsilon}_h-u^{\varepsilon}
  \end{aligned}
 \end{equation}
and for later purposes we set  
\begin{equation}
  \begin{aligned}
   \bar A^{\varepsilon}_h =& D^2f_{\varepsilon}(\nabla u^{\varepsilon}) \\
   \bar a^{\varepsilon}_h=& \frac{1}{k}f_{\varepsilon}(\nabla u^{\varepsilon})^{\frac{1}{k}-1}
   Df_{\varepsilon}(\nabla u^{\varepsilon}).
  \end{aligned}
 \end{equation}

We define $\varphi \in W^{1,p^{*}}_0(\Omega)$ by
\beq \label{73}
L^{*}_{\varepsilon}\varphi = |e^{\varepsilon}_h|^{p-1}\sgn (e^{\varepsilon}_h)
\eeq
and let $\varphi_h \in V_h$ be the finite element solution of this equation. 
We test (\ref{73}) with $e^{\varepsilon}_h$ and get in view of the symmetry of 
$\bar A^{\varepsilon}_h$ that
\begin{equation} \label{7001}
 \begin{aligned}
  \int_{\Omega}|e^{\varepsilon}_h|^p dx=&
   \int_{\Omega}(A^{\varepsilon}_h\nabla e^{\varepsilon}_h) \cdot \nabla \varphi_h dx + 
 \int_{\Omega}(a^{\varepsilon}_h\cdot \nabla e^{\varepsilon}_h) \varphi_h dx \\
 \stackrel{(\ref{72})}{=}& \int_{\Omega}\left((A^{\varepsilon}_h- \bar A^{\varepsilon}_h)\nabla e^{\varepsilon}_h \right)\cdot \nabla \varphi_h dx + 
 \int_{\Omega}\left((a^{\varepsilon}_h-\bar a^{\varepsilon}_h)\cdot \nabla e^{\varepsilon}_h\right) \varphi_h dx \\
 \le& c\int_{\Omega}|\nabla e^{\varepsilon}_h|^2|\nabla \varphi_h|dx + c \int_{\Omega}|\nabla e^{\varepsilon}_h|^2
 \varphi_h dx \\
 \le& c\|\varphi_h\|_{W^{1,p^{*}}(\Omega)} \|e^{\varepsilon}_h\|^2_{W^{1,2p}(\Omega)}.
 \end{aligned}
\end{equation}
In view of Corollary \ref{5004} and (\ref{73}) we get
\begin{equation}
\begin{aligned}
 \|\varphi_h\|_{W^{1,p^{*}}(\Omega)} \le & c \left(\int_{\Omega}|e^{\varepsilon}_h|^{(p-1)p^{*}}
 dx\right)^{\frac{1}{p^{*}}}
 =& c\|e^{\varepsilon}_h\|^{\frac{p}{p^{*}}}_{L^p(\Omega)},
\end{aligned}
\end{equation}
so that
\beq
\|e^{\varepsilon}_h\|_{L^p(\Omega)} \le c h^2.
\eeq
Here, we used (\ref{7001}) and Theorem \ref{60} to estimate 
$\|e^{\varepsilon}_h\|_{W^{1,2p}(\Omega)}$. 
\end{proof}

In the following we validate the convergence rates of Theorem \ref{60} and Theorem~\ref{87}
with a numerical example.
Figure \ref{Fig1} shows the discretization error in the case of a unit circle as initial curve and 
$\varepsilon=0.1$ fixed. The discrete solutions for different values of the 
discretization parameter $h$ are compared with the discrete solution $\tilde u^{\varepsilon}$ 
on a fine grid with grid size $h=0.005$. We compare on $\Omega_{0.005}$ and extend
$u_h^{\varepsilon}=0$, where it is not defined.
The discretization errors 
$\|\tilde u^{\varepsilon}-u^{\varepsilon}_h\|_{\Omega_h}$ for
$\|\cdot \|_{\Omega_h}=\|\cdot \|_{{L^2}(\Omega_h)}$, $\|\cdot \|_{\Omega_h}=\|\cdot \|_{{H^1(\Omega_h)}}$ and 
$\|\cdot \|_{\Omega_h}=
\|\cdot \|_{L^{\infty}(\Omega_h)}$ are plotted and behave as shown in Theorem \ref{87} and Theorem 
\ref{60}.

\begin{figure}[htb] 
    \centering
      {
        \includegraphics[width=0.47\textwidth]{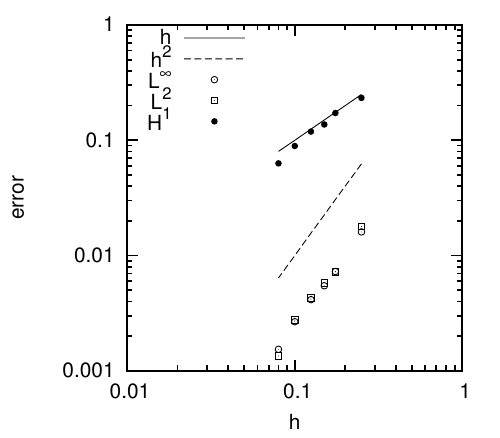}
        }
%

\caption{Discretization error for the unit circle as initial curve where $k=1$, $\varepsilon=0.1$.}
\label{Fig1}
\end{figure}


In Figure \ref{Fig8} we see that the discretization error close to the boundary has a pattern
which results from the approximation of the smooth domain $\Omega$ by the polygonal 
domain $\Omega_h$.
The discrete solution $u^{\epsilon}_h$ is zero at the polygonal boundary $\partial \Omega_h$
while $u^{\epsilon}$ is zero on the curved boundary $\partial \Omega$.

\begin{figure}
  \centering
 \includegraphics[width=7cm]{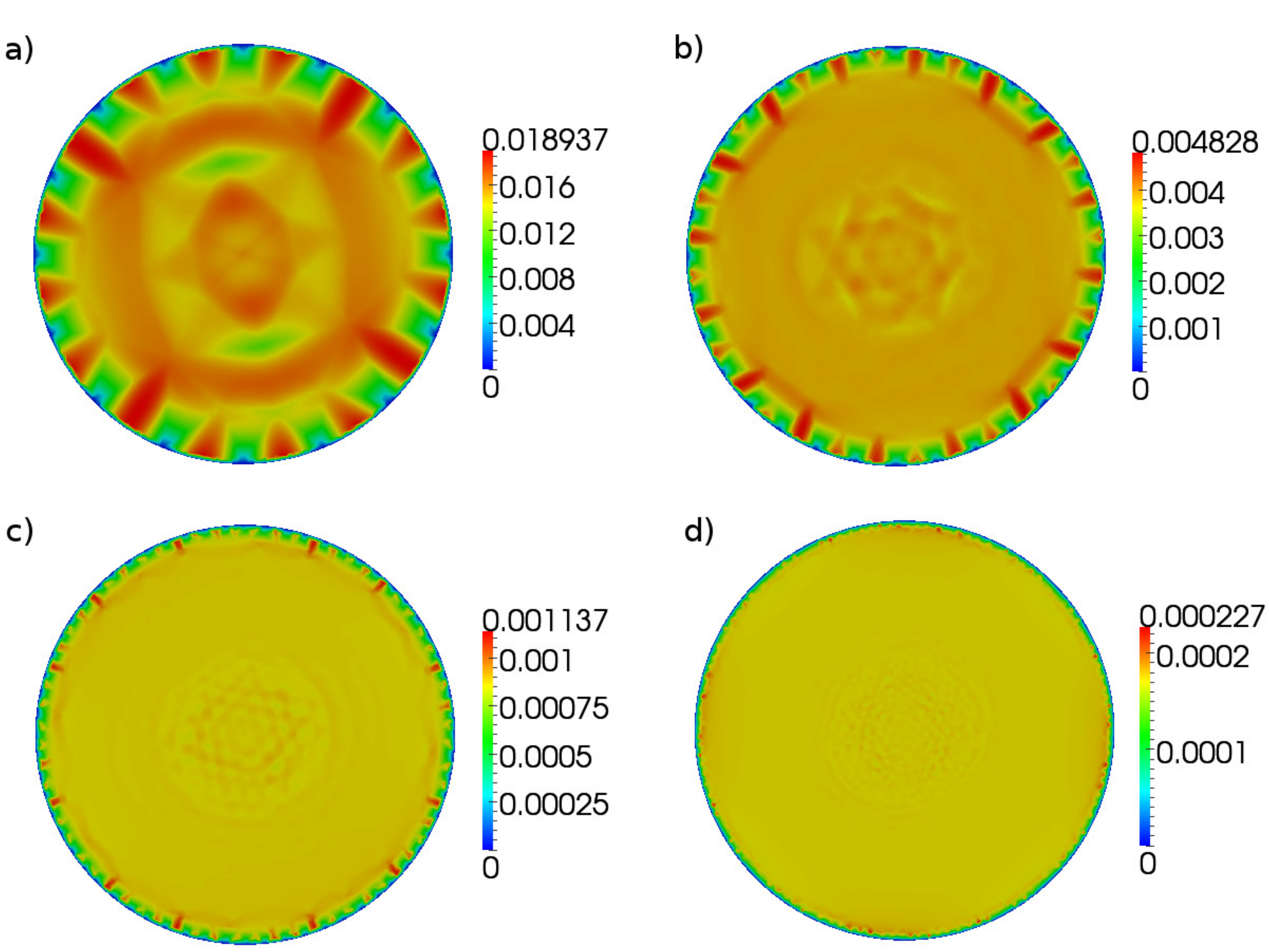}
 \caption{Discretization error for the unit circle as initial curve (we compare with a discrete
 solution on a fine grid with size $h_0$). The pictures show the difference 
 $u^{\varepsilon}_{h_0}-u^{\varepsilon}_{h}$  for $\varepsilon = 0.09$, $k = 1$, $h_0 = 0.025$ and  a) $h=0.4$, b) $h=0.2$, c) $h=0.1$, d) $h=0.05$.}
 \label{Fig8}
 \end{figure}
 
\section{Regularization error} \label{50}

We specify our a priori estimate for the regularization error presented in \cite{K} which depends on the choice of certain constants. Let
\beq \label{9}
\ga >1+k
\eeq
and $\al, s>0$ so that
\beq \label{10}
\be_1(\al, s) > \be_2(\al, s),
\eeq
where 
\begin{equation}  \label{23_}
\begin{aligned} 
\be_1(\al, s) :=\frac{2-s+\al(2-\frac{1}{k})}{\ga(2-\frac{1}{k})+\frac{1}{k}-1}, \quad 
\be_2(\al, s):=\frac{\al+ks}{\ga-k-1}
\end{aligned} 
\end{equation}
and choose 
\beq \label{714}
0<r<\frac{\al}{\ga}.
\eeq 
Note, that (\ref{10}) obviously holds for sufficiently small $\al, s$.
There holds the following theorem, cf. \cite[Theorem 3.1]{K}.
\begin{theorem} \label{24}
There is $c=c(k, \Omega)>0$ such that
\beq
\|u^{\varepsilon}-u\|_{C^0(\bar \Omega)} \le c \varepsilon^{\min(r,s)}
\eeq
for all $\varepsilon>0$.
\end{theorem}

The order of convergence stated in the previous theorem can be written more explicitly which is content of the 
following lemma.
\begin{corollary} \label{76}
 For the evolution with normal speed $H^k$, $k \ge 1$, the regularization error with respect 
 to the $C^0$-norm is of order 
 $O(\varepsilon^{\frac{1}{\la}})$ for all $\la > 2k$.
\end{corollary}
\begin{proof}
We rewrite the right-hand side of the estimate in Theorem \ref{24}. In (\ref{10}) we may assume that
$\al, s$ are related by $s=\frac{\al}{\ga}$.
We multiply (\ref{10}) by $\gamma$ and get
\beq
\frac{2-s+\al(2-\frac{1}{k})}{2-\frac{1}{k}+\frac{\frac{1}{k}-1}{\ga}} > \frac{\al+sk}{1-\frac{k+1}{\ga}}.
\eeq 
Now, we multiply with the denominator of the left-hand side and sort by $\al$ and $s$ on each side which 
leads to
\beq
2-s+\al\left(2-\frac{1}{k}\right) > \al \frac{2-\frac{1}{k}}{1-\frac{k+1}{\ga}}
+ s \frac{\frac{1}{k}-1+k\left(2-\frac{1}{k}+\frac{\frac{1}{k}-1}{\ga}\right)}{1-\frac{k+1}{\ga}}
\eeq
and after rearranging terms to
\begin{equation} \label{93}
\begin{aligned}
2  >& 
s\frac{\frac{1}{k}-\frac{k+1}{\ga}+k\left(2-\frac{1}{k}+\frac{\frac{1}{k}-1}{\ga}\right)
+(2-\frac{1}{k})(k+1)}{1-\frac{k+1}{\ga}}. 
\end{aligned}
\end{equation}
We may let $\ga$ tend to infinity without changing the value of $s$ (by adapting $\al$ correspondingly). Hence
the right-hand side of (\ref{93}) converges to $4ks$ as $\ga\rightarrow \infty$ and the claim follows.
\end{proof}

We recall an interpolation lemma, cf \cite[PDE II, Lemma 1.4.13]{CG}.

\begin{lemma}\label{interpolation}
For $0<\be<\al\le 1$ and a function $v: \Omega \rightarrow \mathbb{R}$ holds
\beq
[v]_{\be} \le 2^{1-\frac{\be}{\al}}[v]_{\al}^{\frac{\be}{\al}}\|v\|_{C^0(\Omega)},
\eeq
where these expressions might become $\infty$ and 
\beq
[v]_{\al} = \sup_{x \neq y}\frac{|v(x)-v(y)|}{|x-y|^{\al}}.
\eeq
\end{lemma}
Since $u^{\varepsilon}$ is uniformly bounded in the $C^1$-norm, cf. (\ref{12}), we can use Lemma \ref{interpolation} to get
\beq
\|u-u^{\varepsilon}\|_{C^{0,\beta}(\Omega)} \le c (\be) \varepsilon^{\la(1-\be)}
\eeq
for every $0<\be<1$ and $0<\la<\frac{1}{2k}$.

In the case $k=1$ which means mean curvature flow we can realize in Corollary \ref{76} every power of 
$\varepsilon$ which lies in $(0, \frac{1}{2})$. 
This is in accordance with the corresponding rate for the case of the 
time dependent 
level set regularization as considered in 
Deckelnick's paper \cite[Theorem 1.2]{D} and Mitake's paper \cite[Theorem 1]{M}.

The following numerical examples indicate that the regularization error is even smaller than stated in 
Corollary 
\ref{76}.
As a first example we calculate the regularization error in the case of the 
evolution of a unit circle as initial curve for which the exact solution of 
equation (\ref{levelset_pmcf}) is known. Let $\partial B_{r_0}(0)\subset \mathbb{R}^2$, 
i.e. a circle with radius $r_0>0$, be the initial curve then the exact solution $u$ is given as 
\begin{equation}
u(r) = \frac{r_0^{k+1}-r^{k+1}}{k+1},
\end{equation} 
where $r$ denotes the radius variable in polar coordinates in $\mathbb{R}^2$ with center in $0$. 
As special case we choose $k=1$, $r_0=1$, i.e. 
\begin{equation}
u(r) = \frac{1}{2}-\frac{r^2}{2}.
\end{equation}
In Figure \ref{Fig2} the error $\|u^{\varepsilon}-u\|$ is plotted in this special case 
(and for $k=1.5$ and $k=2$), where $\|\cdot \|$ stands 
for $\|\cdot \|=\|\cdot \|_{L^2(\Omega_h)}$, 
$\|\cdot \|=\|\cdot \|_{H^1(\Omega_h)}$ and $\|\cdot \|=\|\cdot \|_{L^{\infty}(\Omega_h)}$. 
We remark that our theoretical estimate in Corollary \ref{76} does not provide information about an estimate
with respect to $\|\cdot \|_{H^1(\Omega_h)}$.
The functions $u^{\varepsilon}$ are calculated by using linear finite elements on a 
 fine grid with mesh size $h=0.0125$. The $L^2$-error converges a little bit faster and the $H^1$-error a little bit slower 
than of quadratic order to zero. 

\begin{figure}
  \begin{center}
 \includegraphics[width=6.23cm]{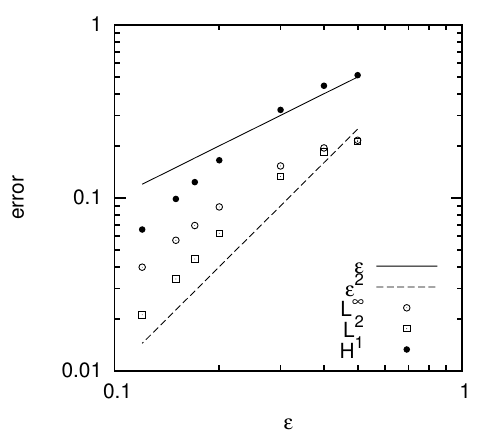}
 \includegraphics[width=6.23cm]{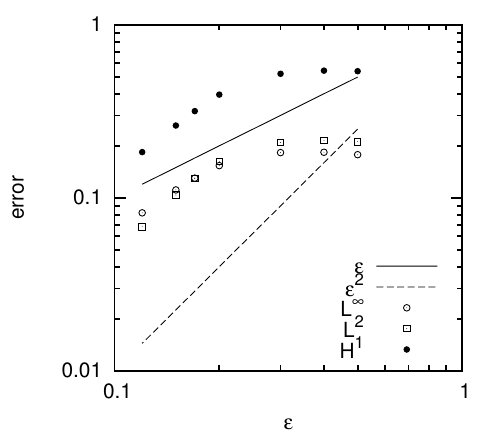}
 \includegraphics[width=6.23cm]{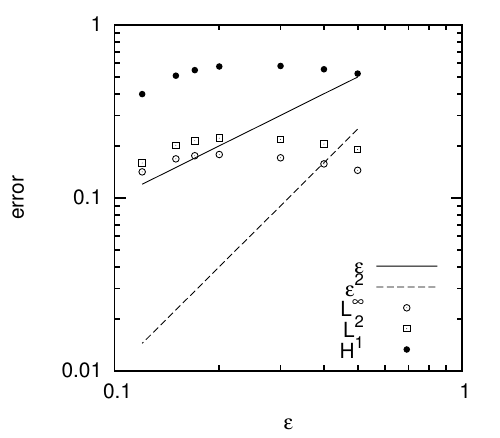}
 \caption{Regularization error in case of a circle as initial curve for $k=1, 1.5, 2$.}
 \label{Fig2}
 \end{center}
 \end{figure}
 
In Figure \ref{Fig3} the scenario is the same as in Figure \ref{Fig2} apart from the fact that we now choose 
an ellipse (half axes lengths 1 and 2) as initial curve. Furthermore since we do not have an exact solution 
$u$ for this case we use instead a solution $u^{\varepsilon}_h$ with
  $\varepsilon=0.1$ and small $h=0.0125$. 
  In Figure \ref{Fig4} we plot a section (along the long and short half axes of the initial curve)
  of the solution $u^{\varepsilon}$ in the case 
  of the circle and in Figure \ref{Fig5} in the case of an ellipse as initial curve
  for different values of $\varepsilon$.
  
 In accordance with our a priori estimate in Corollary \ref{76} the regularization error in Figure \ref{Fig2} seems 
 to become larger for increasing $k$. This can be also seen from Figure~\ref{Fig4}. There we also 
observe that the approximation quality around the singularity of the flow deteriorates for $k=2$
when changing $\varepsilon$ from $0.5$ to $0.17$.

\begin{figure}
  \begin{center}
 \includegraphics[width=6.23cm]{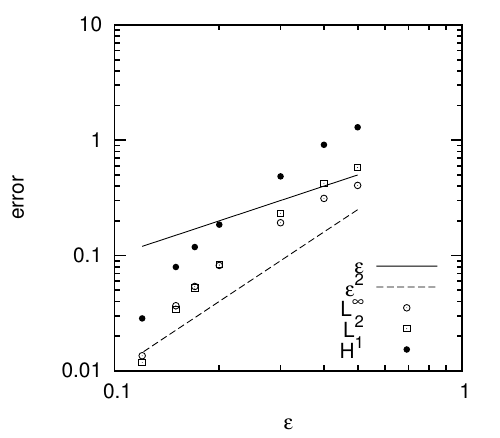}
 \includegraphics[width=6.23cm]{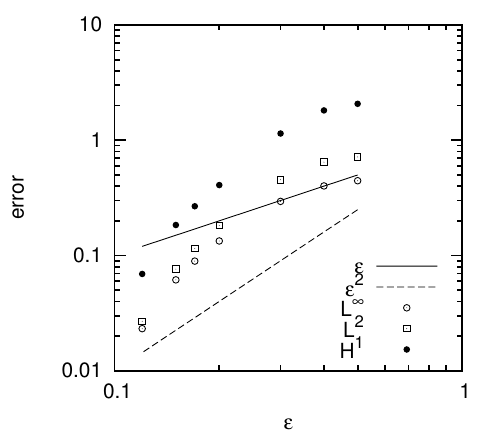}
 \includegraphics[width=6.23cm]{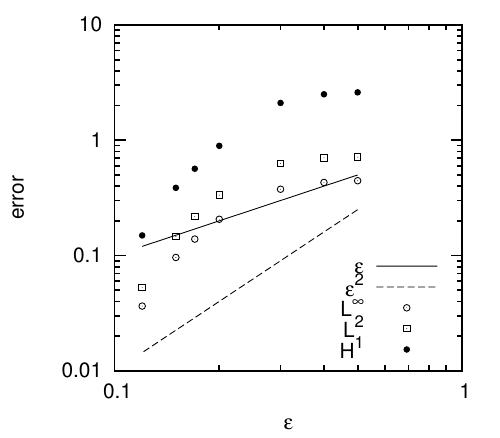}
 \end{center}
 \caption{Regularization error in case of an ellipse as initial curve for $k=1, 1.5, 2$.}
 \label{Fig3}
 \end{figure}

\begin{figure}
  \begin{center}
 \includegraphics[width=6.23cm]{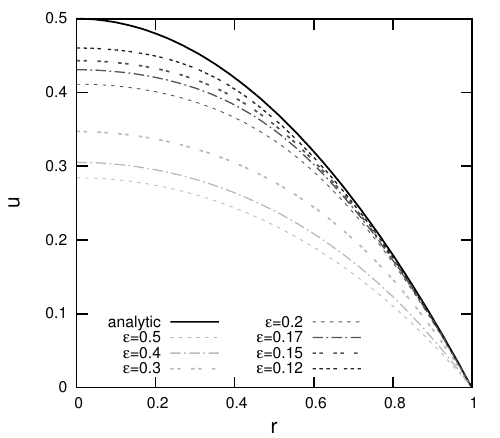}
 \includegraphics[width=6.23cm]{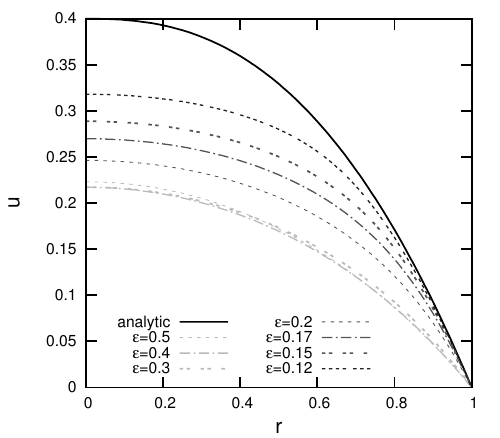}
 \includegraphics[width=6.23cm]{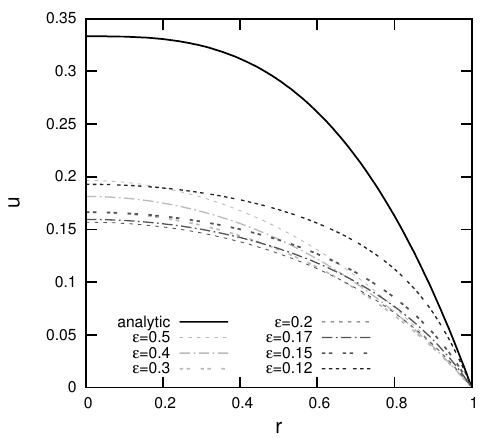}
 \caption{Radial solution for a circle as initial curve for $k=1, 1.5, 2$.}
 \label{Fig4}
 \end{center}
 \end{figure}

\begin{figure}
\begin{center}
\includegraphics[width=6.23cm]{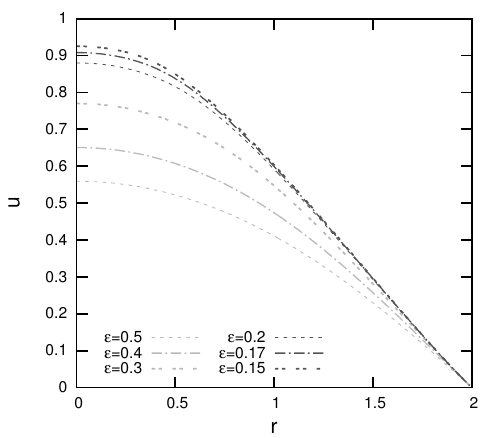}
 \includegraphics[width=6.23cm]{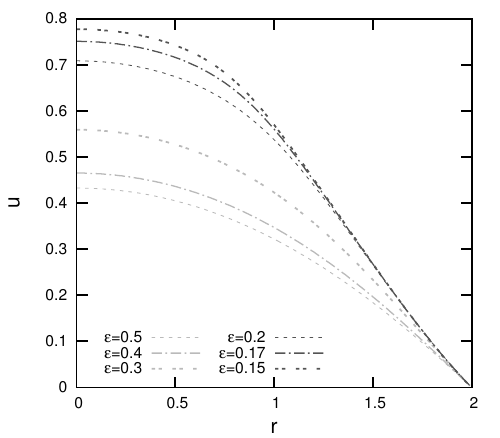}
 \includegraphics[width=6.23cm]{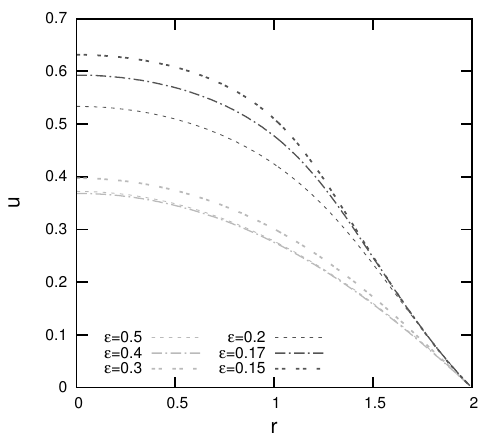}\\
  \includegraphics[width=6.23cm]{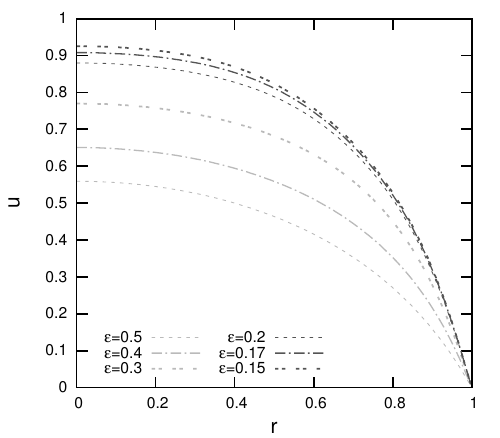}
 \includegraphics[width=6.23cm]{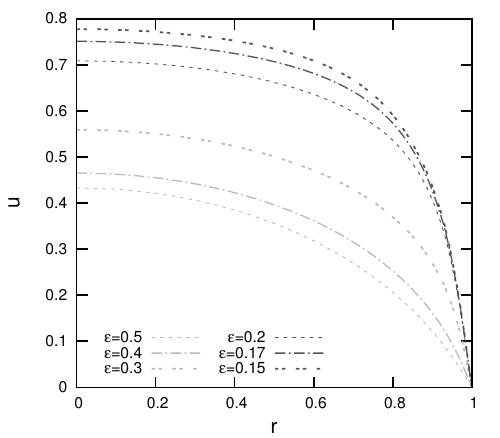}
 \includegraphics[width=6.23cm]{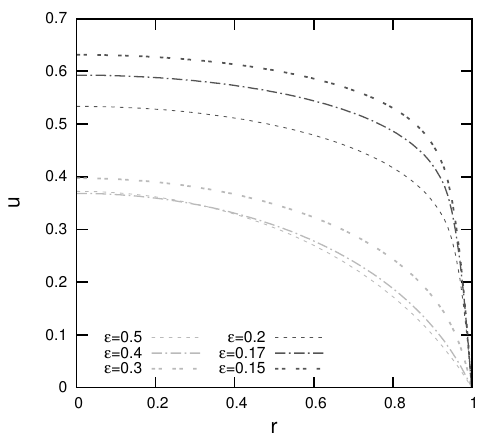}
 \caption{Solution for ellipse as initial curve. Picture 1--3: Section in direction of the long half axis of the initial curve
 for $k=1, 1.5, 2$; 
 Picture 4--6: the same for the short half axis.}
 \label{Fig5}
 \end{center}
 \end{figure}

Figure \ref{Fig9} shows level sets of $u^{0.1}$ for the case of the ellipse as initial curve and different values of 
$k$. We remark that our theory covers only the case $k\ge 1$ but in the special case of convex curves we also have a 
level set solution for general $k>\frac{1}{3}$ which follows from the classification
of the behavior of the evolution of curves by powers of the curvatures presented in Section 1 of \cite{A03}. 
Our observations are as follows. For $k=0.5$ we see for $\varepsilon=0.1$ a quite good approximation
of the phenomenon of shrinking to a 'round point' and further lessening of $\varepsilon$ does not show significant 
improvements. For all $k$ the inner level line for $\varepsilon=0.1$ seems to be already 'round' while for $k=2$ this seems 
to be  
far from a 'point'.
 
\begin{figure}
  \centering
 \includegraphics[width=6cm]{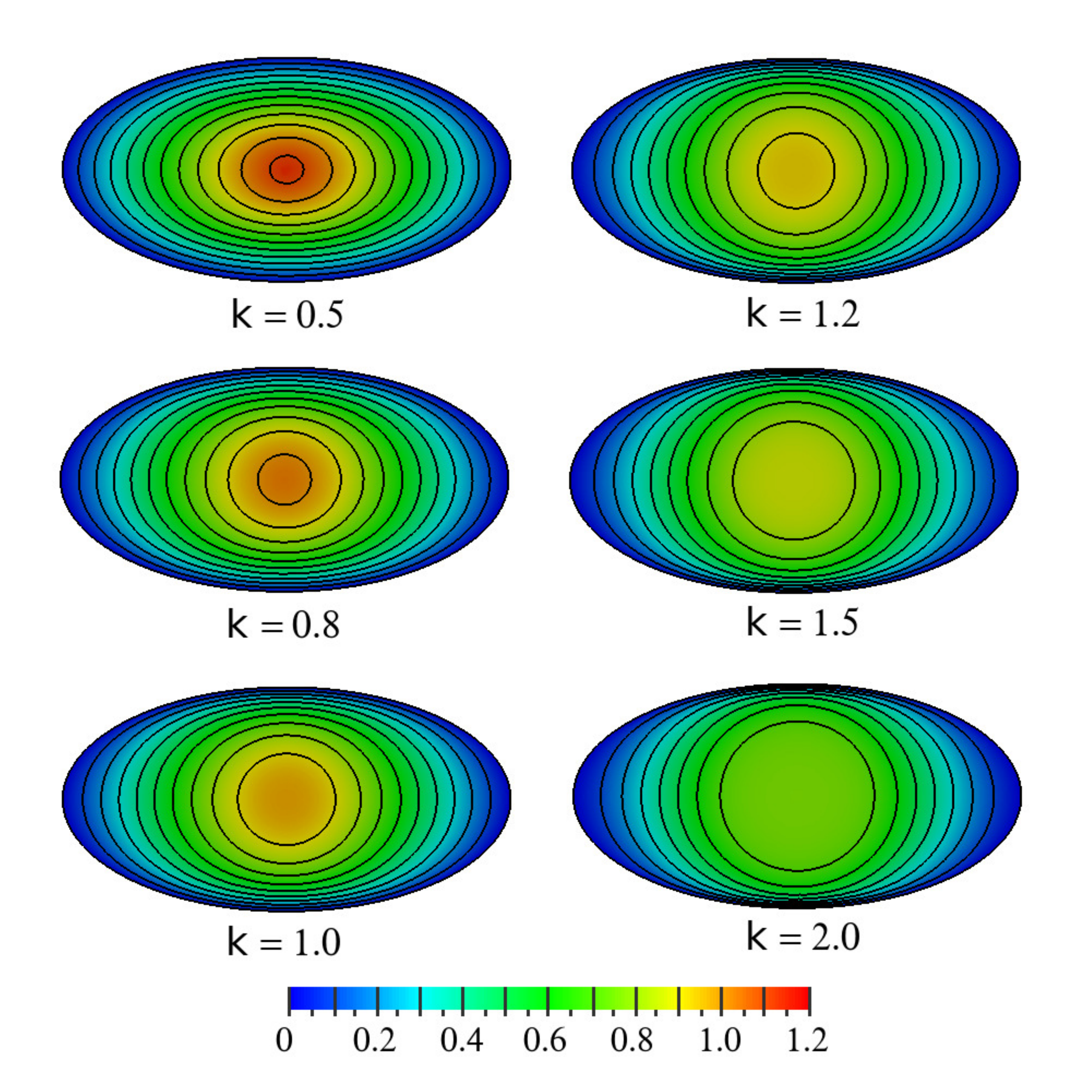}
 \caption{Solution for ellipse for $\varepsilon=0.1$.}
 \label{Fig9}
 \end{figure}

\section{Total approximation error} \label{52}
In the inequalities (\ref{1002}) and (\ref{88}) appear constants $c$ on the right-hand
sides  which depend on the solution $u^{\varepsilon}$ of the regularized equation. 
To get an estimate for $u-u^{\varepsilon}_h$ in terms of $\varepsilon$ and $h$ one has to make this dependence explicit. 
In our paper \cite{K} we showed that there is a $\ga>0$ such that if we couple $h$ and $\varepsilon$ by 
$h=\varepsilon^{\ga}
$ and use finite elements of order 2 (and quadratic boundary approximations), then there holds that the error $u-u^{\varepsilon}_h$ converges to zero with a polynomial
 rate in $h$ with respect to the sup-norm. To estimate with respect to the sup-norm is natural since $u$ is (only, in 
 general,) a $C^0$-limit of $u^{\varepsilon}$ and (as viscosity solution)
 Lipschitz continuous. A value for $\ga$
 and the convergence rate
 can be obtained by adapting it at each stage of the proofs in \cite{K} as described there which leads to
 a rather technical large value of no practical interest.  The main point is that we have a 
 polynomial rate and not an exponential rate. As said before and explained by comparing the
 situation with \cite[Theorem 6.4]{DDE}, where the authors 
 proved even only an exponential estimate, this is non-trivial. We let us furthermore inspire from the
 scenario of \cite[Theorem 6.4]{DDE} which overestimates the error rate as practical results indicate, 
 cf. \cite{DDE}.
 Therefore we start our calculations with the 
 from practical point of view comfortable setting of continuous and piecewise linear 
 finite elements, a polygonal boundary 
 approximation and a coupling between $\varepsilon$ and $h$ by setting $\varepsilon = h$ which already lead to convergence.
 
 Figure \ref{Fig6} shows the total approximation errors for the unit circle as initial curve in the cases $k=1, 1.5, 2$. Although
 we have only for the sup-norm a theoretical estimate we also plot the $H^1$-error; we remark
 that in the 
 situation of the circle the solution is of class $C^{\infty}(\bar B_{r_0}(0))$.
 Figure \ref{Fig7} shows the same scenario as Figure \ref{Fig6} apart from the fact that we now consider the ellipse 
(half axes with lengths 1 and 2) as initial curve.
Furthermore, as reference solution we consider a solution with $h=0.05$
and $\varepsilon=0.05$.

\begin{figure}
\begin{center}
 \includegraphics[width=6.23cm]{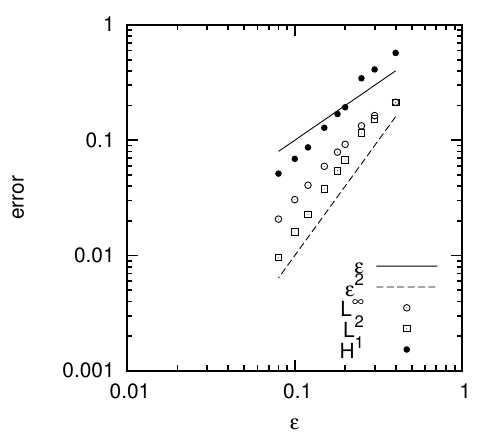}
 \includegraphics[width=6.23cm]{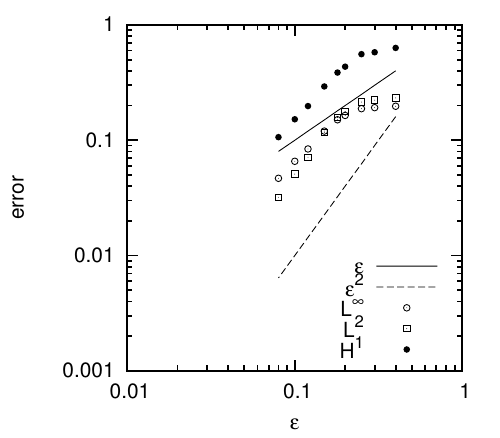}
 \includegraphics[width=6.23cm]{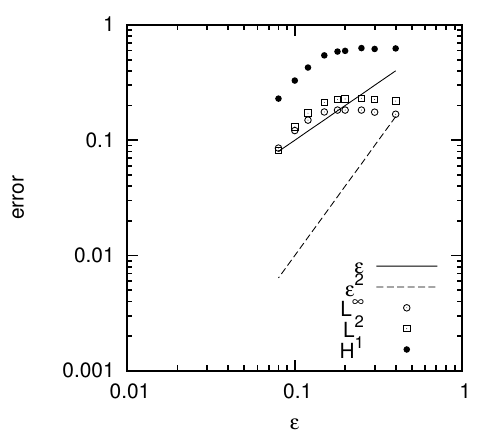}
 \caption{Total approximation error, $\varepsilon = h$, for $k=1, 1.5, 2$, in case of a circle as initial curve.}
 \label{Fig6}
 \end{center}
 \end{figure}
 
\begin{figure}
\begin{center}
 \includegraphics[width=6.23cm]{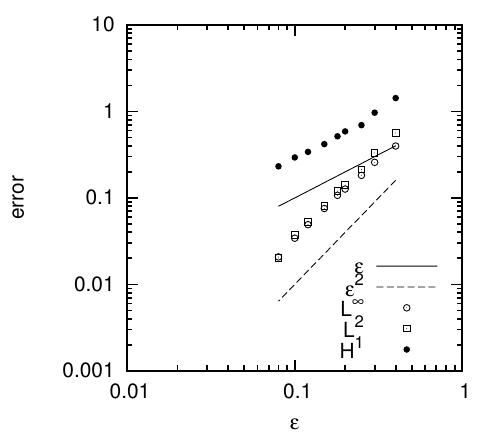}
 \includegraphics[width=6.23cm]{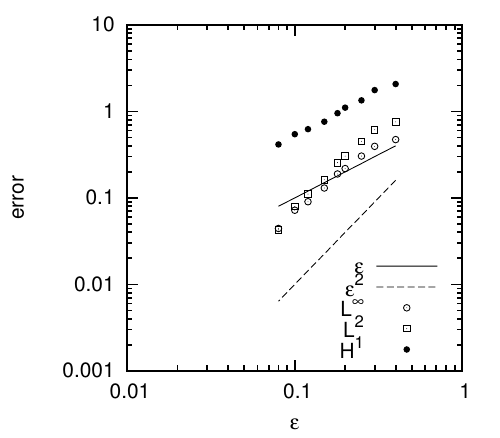}
 \includegraphics[width=6.23cm]{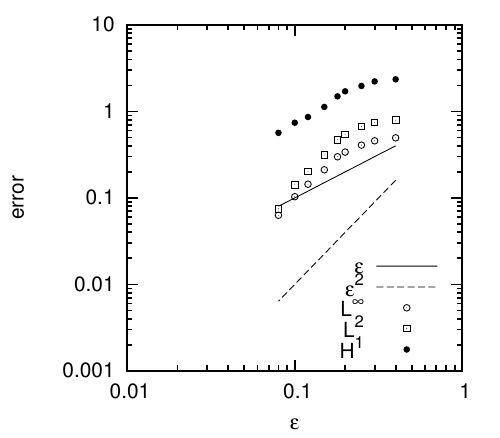}
 \caption{Total approximation error, $\varepsilon = h$, for $k=1, 1.5, 2$ in case of ellipse with half axes 1 and 2 as initial curve.}
 \label{Fig7}
 \end{center}
 \end{figure}
 
\section{Effect of k on behavior of the flow for an example case} \label{70}
The phenomenon of becoming round can be measured by the isoperimetrical deficit
\beq
l(t)^2 - 4 \pi a(t), 
\eeq
where $l(t)$ denotes the length of the curve and $a(t)$ the enclosed area at time $t$. According to theoretical results
in \cite{S} we confirm the monotonicity of this deficit during the evolution in the special case of the 
ellipse as initial curve, see Figure \ref{Fig10}. Furthermore, we see that with increasing $k$ (and $\varepsilon=0.05$) the 
curves transform faster into a circle (they are not yet shrinked to a point except for $k=0.5$, see Figure \ref{Fig9},
where the 'point' is reached quite well). In Figure \ref{Fig4} and Figure \ref{Fig5} we see 
when comparing the exact solutions for the circle for different values of $k$ and the approximate 
solutions for the ellipse with $\epsilon=0.15$ for different values of $k$, respectively, that the flow reaches the 
singularity for larger $k$ earlier.

\begin{figure}
  \centering
 \includegraphics[width=6cm]{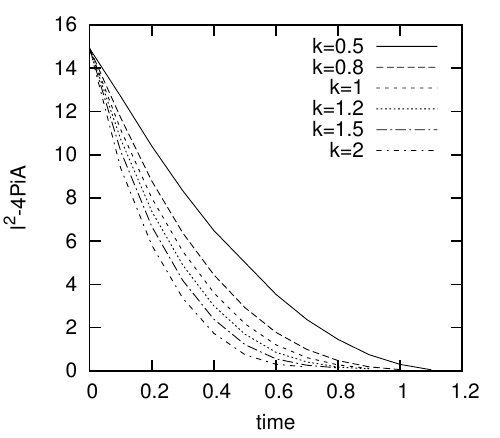}
 \caption{Isoperimetrical deficit.}
 \label{Fig10}
 \end{figure}
 
\section{Implementation}
To calculate the finite element approximation $u^{\varepsilon}_h$ of $u^{\varepsilon}$ we used a discretization
with unstructured grids, see Figure  \ref{Fig11}. These were generated by the mesh generator Gmsh, see \cite{Gmsh}. 
We solved the non-linear equation (\ref{FE_levelset_pmcf}) with a Newton method which uses a 
bi-conjugate gradient stabilized solver (BiCGSTAB) and SSOR preconditioning. 
For the implementation we used PDELab, a discretization module for solving PDEs which depends on the Distributed
and Unified Numerics Environment (DUNE). As further references concerning PDELab we refer to \cite{PDELab, Bastian3}, 
information 
about DUNE can be found in \cite{Blatt, Bastian1, Bastian2, DUNE}.
In order to get solutions for small $\varepsilon$ we used a warm-start, i.e. we decreased $\varepsilon$ stepwise to the desired
small value and performed on each stage a calculation with the solution for the previous $\varepsilon$
as initial value.

\begin{figure}
  \centering
 \includegraphics[width=6cm]{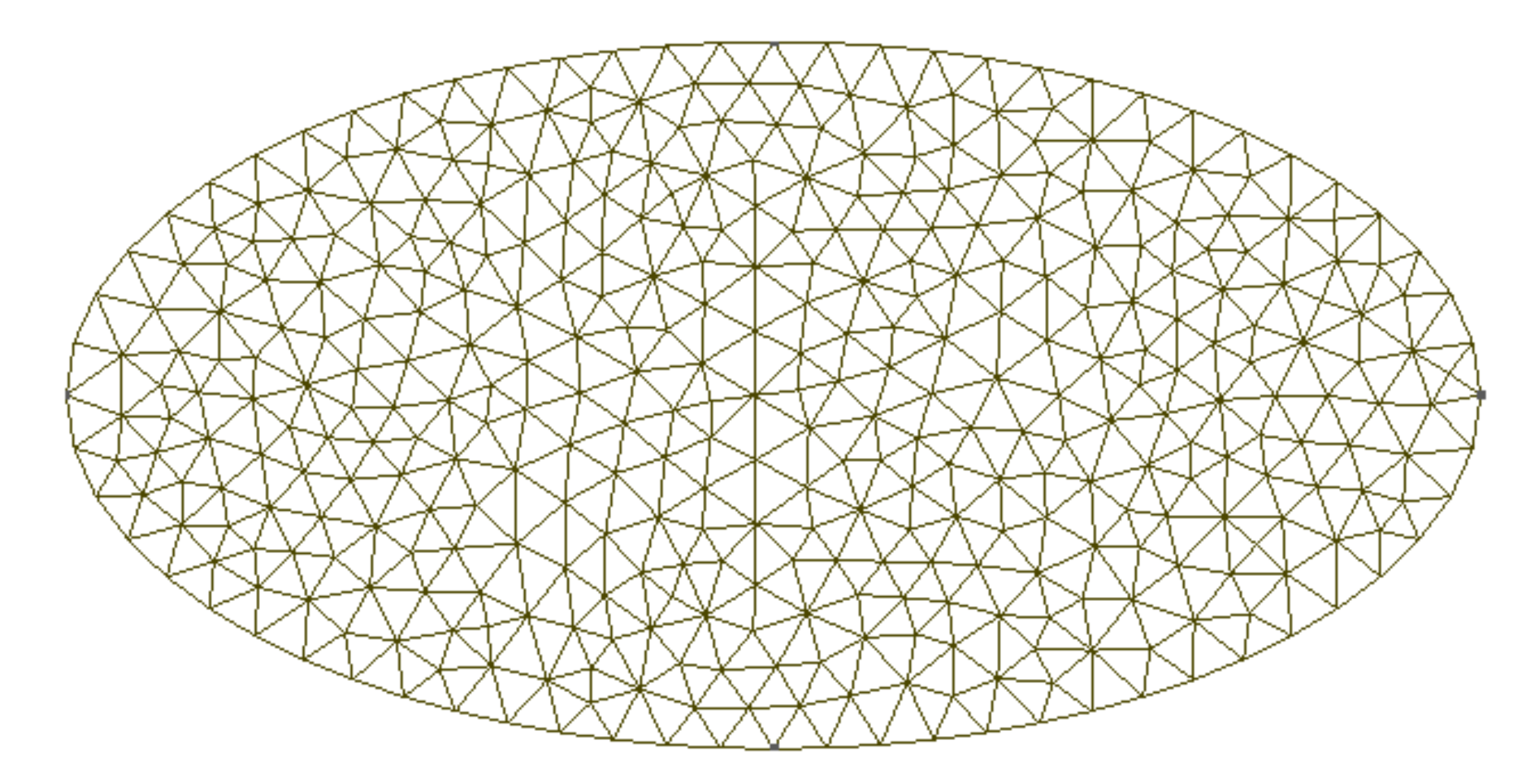}
 \caption{Mesh for the discretization with size $h=0.15$ for ellipse with half axes 1 and 2.}
 \label{Fig11}
 \end{figure}
 \section{Appendix}
 Since $L_{\epsilon}: H^1_0(\Omega)\rightarrow H^{-1}(\Omega)$ is a topological isomorphism by classical $L^2$-theory this
also holds for $L_{\epsilon}^{*}: H^1_0(\Omega) \rightarrow H^{-1}(\Omega)$.
We define the to $L_{\epsilon}$ associated uniformly, elliptic regular Dirichlet bilinear form of order 1 by
\beq
B: W^{1,p}_0(\Omega)\times W^{1, p^{*}}_0(\Omega)\rightarrow \mathbb{R}, 
\quad B[u,v]=\int_{\Omega}a^{ij}D_iuD_jv + b^iD_iuv \ dx
\eeq
and set
\begin{equation}
\begin{aligned}
N_{p^{*}} =& \{v \in W^{1,p^{*}}_0(\Omega): B[\psi, v]=0 \text{ for every }\psi \in C^{\infty}_0(\Omega)\}\\
N_{p} =& \{v \in W^{1,p}_0(\Omega): B[v, \phi]=0 \text{ for every }\phi \in C^{\infty}_0(\Omega)\}
\end{aligned}
\end{equation}
From Fredholm's alternative, cf. \cite[Theorem 10.7]{Simader1972}, we deduce that for every $F \in W^{-1,p^{*}}(\Omega)$
the equation
\beq \label{5003}
B[u, \varphi] = F\varphi \quad\forall \varphi \in W^{1,p^{*}}_0(\Omega)
\eeq
has a solution $u\in W^{1,p}_0(\Omega)$ if and only if 
\beq \label{5001}
v\in N_{p^{*}} \Rightarrow Fv=0.
\eeq
If $\dim N_{p^{*}}=\dim N_{p}=0$ then for every $F\in W^{-1, p^{*}}(\Omega)$ equation (\ref{5003}) has a unique solution.
\begin{lemma}
$\dim N_{p^{*}}=\dim N_{p}=0$. 
\end{lemma}
\begin{proof}
Let $v \in N_{p^{*}}$. From \cite[Theorem 7.6]{Simader1972} we get $v\in W^{1,p'}_0(\Omega)$ for all $1<p'<\infty$, especially
for $p'=2$. Since we know from $L^2$-theory that (\ref{5003}) has a unique solution $u \in W^{1,2}_0(\Omega)$ if $p=2$ and $F=0$ we deduce  that $v=0$. Analogously we obtain the remaining claim. 
\end{proof}

By bounded inverse theorem we conclude the following result.
\begin{corollary} \label{5004}
$L_{\epsilon}, L_{\epsilon} ^{*}$ are topological isomorphisms.
\end{corollary}

\section*{Acknowledgment}
We thank Klaus Deckelnick and Ulrich Matthes for a discussion on the question of 
dimensionality for the $L^p$-estimates of the finite element solution.

The work of this paper was partly carried out while the 
third author benefited from a Weierstrass postdoctoral fellowship of the Weierstrass 
Institute Berlin.

\end{document}